\documentclass[11pt,reqno]{amsart}
\usepackage{amsmath,amssymb,amsthm,amsfonts,enumerate,hyperref}
\usepackage{verbatim}

\theoremstyle{plain}
\newtheorem{theorem}{Theorem}[section]
\newtheorem{theorem2}{Theorem}[section]
\newtheorem{proposition}[theorem2]{Proposition}
\newtheorem{cor}[theorem2]{Corollary}
\newtheorem{lemma}[theorem2]{Lemma}

\theoremstyle{definition}
\newtheorem{definition}[theorem2]{Definition}

\theoremstyle{remark}

\numberwithin{equation}{section}

\newcommand{\Riem}{\mathrm{Rm}}

\newcommand{\vol}{\mathrm{Vol}}

\newcommand{\Ric}{\mathrm{Ric}}

\newcommand{\Hess}{\mathrm{Hess}}
\newcommand{\tr}{\mathrm{tr}}

\renewcommand{\div}{\mathrm{div}}

\begin{document}
\title{Compact Hermitian symmetric spaces, coadjoint orbits, and the dynamical stability of the Ricci flow}
\author{Stuart James Hall}
\address{School of Mathematics and Statistics, Herschel Building, Newcastle University, Newcastle-upon-Tyne, NE1 7RU} 
\email{stuart.hall@ncl.ac.uk}
\author{Thomas Murphy}
\address{Department of Mathematics, California State University Fullerton, 800 N. State College Bld., Fullerton, CA 92831, USA.}
\email{tmurphy@fullerton.edu}
\author{James Waldron}
\address{School of Mathematics and Statistics, Herschel Building, Newcastle University, Newcastle-upon-Tyne, NE1 7RU} 
\email{james.waldron@ncl.ac.uk}
\maketitle  \vspace{-10pt}

\begin{abstract}
Using a stability criterion due to Kr\"oncke, we show, providing ${n\neq 2k}$, the K\"ahler--Einstein metric on the Grassmannian $Gr_{k}(\mathbb{C}^{n})$ of complex $k$-planes in an $n$-dimensional complex vector space is dynamically unstable as a fixed point of the Ricci flow. This generalises the recent results of Kr\"oncke and Knopf--Sesum on the instability of the Fubini--Study metric on $\mathbb{CP}^{n}$ for $n>1$. The key to the proof is using the description of Grassmannians as certain coadjoint orbits of $SU(n)$. We are also able to prove that Kr\"oncke's method will not work on any of the other compact, irreducible, Hermitian symmetric spaces. 
\end{abstract}
\section{Introduction} \label{sec:1}
In 2013 Kr\"oncke proved the surprising result that the Fubini--Study K\"ahler--Einstein metric on $\mathbb{CP}^{n}$, $n>1$,  is unstable as a fixed point of the Ricci flow \cite{KlKr1}. More precisely, he showed that there are certain conformal (and hence non-K\"ahler) deformations of the Fubini--Study metric from which the Ricci flow never returns. This is in stark contrast to the behaviour of the K\"ahler--Ricci flow where Tian and Zhu \cite{TZ} have shown that K\"ahler--Einstein metrics are essentially global attractors within their K\"ahler class. In \cite{KS} Knopf and Sesum give an independent verification of Kr\"oncke's result.\\
\\
The behaviour of Ricci flow on manifolds admitting K\"ahler metrics is a topic of current interest (see for example \cite{HMPJM}, \cite{HMAGAG1}, \cite{IKS}, and \cite{Max}). What Kr\"oncke's result suggests is that behaviour of the Ricci flow near the space of K\"ahler metrics is more complicated than was initially believed. If a Fano manifold $M$ with Hodge number $h^{1,1}(M)>1$  admits a K\"ahler--Einstein metric then it can be destabilised by a harmonic perturbation within the K\"ahler cone. This method can be used to show many known examples of K\"ahler--Einstein metrics are unstable. However, as the complex dimension of the Fano manifold grows, there are numerous examples of K\"ahler--Einstein manifolds with $h^{1,1}(M)=1$.  One such class of K\"ahler--Einstein manifolds are the compact, irreducible, Hermitian symmetric spaces. These manifolds were completely classified by E. Cartan into six types; there are four infinite families and two exceptional spaces. Each of these spaces admits a K\"ahler--Einstein metric unique up to automorphisms of the complex structure; this metric is the symmetric metric on each manifold. We will henceforth implicitly assume all manifolds in this paper are equipped with their symmetric space (and K\"ahler--Einstein) metrics.\\
\\
The stability criterion employed by Kr\"oncke is very simple to state (c.f. Theorem \ref{KKstabthm}); if an Einstein metric with Einstein constant $\frac{1}{2\tau}>0$  admits an eigenfunction of the Laplacian, $f$ say, with eigenvalue $-\frac{1}{\tau}$ and the integral over the manifold of $f^{3}$ does not vanish,  then the metric is dynamically unstable. It is a classical result of Matsushima \cite{Mat} that, on a Fano K\"ahler--Einstein manifold, there is a bijection between Killing fields and the eigenspace corresponding to $-\frac{1}{\tau}$.  Hence K\"ahler--Einstein manifolds with large symmetry groups are ideal candidates on which to attempt to use Kr\"oncke's result to investigate stability. The first theorem we prove says that, when the symmetric space is not a Grassmannian of complex $k$-planes in an $n$-dimensional complex vector space (which we denote $Gr_{k}(\mathbb{C}^{n})$), the integral of $f^{3}$ will necessarily vanish.

\begin{theorem}\label{ThmCST}
Let $(M,g)$ be a compact, irreducible, Hermitian symmetric space which is not a Grassmannian $Gr_{k}(\mathbb{C}^{n})$ and let $g$ be the canonical K\"ahler--Einstein metric normalised to have Einstein constant $\frac{1}{2\tau}$. Then any ($-\dfrac{1}{\tau}$)-eigenfunction of the Laplacian, $f$, satisfies
\begin{equation*}
\int_{M}f^{3} \ dV_{g} = 0.
\end{equation*}
\end{theorem}
The proof of this theorem uses the Chevalley--Shepherd--Todd theorem which classifies the degrees in which the generators of certain polynomial algebras can exist when the polynomial is required to be invariant under the action of a compact simple Lie group $G$. Applied to our stability problem, this classification only permits a generator in the required degree when $G=SU(n)$ and so the only class of space where the integral can be non-zero is the complex Grassmannians.\\ 
\\
We are able to compute the stability integral for $Gr_{k}(\mathbb{C}^{n})$ and show that `generic' Grassmannians are unstable. This result generalises the $\mathbb{CP}^{n}$ calculation of Kr\"oncke and Knopf--Sesum.
\begin{theorem}\label{ThmGr}
If $k,n \in \mathbb{N}$ with $1\leq k<n$ and $n\neq 2k$, then  $Gr_{k}(\mathbb{C}^{n})$ is dynamically unstable as a fixed point of the Ricci flow. 
\end{theorem}
The integral of $f^{3}$ vanishes for the spaces $Gr_{k}(\mathbb{C}^{2k})$; the proof of Theorem \ref{ThmGr} shows this directly but it may also  be seen via the following argument. The duality map ${\Psi:Gr_{k}(\mathbb{C}^{n})\rightarrow Gr_{n-k}(\mathbb{C}^{n})}$  which maps a subspace to its orthogonal  complement is an isometry.  In the case when $n=2k$,  $\Psi$ is also an involution. One can show (see \cite{GG}) that  in this situation any ($-\dfrac{1}{\tau}$)-eigenfunction of the Laplacian $f$  satisfies the equation $\Psi^{\ast}f = -f$. This directly implies that the integral of any odd power of $f$ will vanish.\\  
\\
Hence we cannot conclude  anything about the stability of the spaces $Gr_{k}(\mathbb{C}^{2k})$ apart from the case when $k=1$ as then ${Gr_{1}(\mathbb{C}^{2})\cong \mathbb{CP}^{1} \cong \mathbb{S}^{2}}$. In this case, the K\"ahler--Einstein metric is the round metric and this is known to be dynamically stable by a result of Hamilton \cite{Ham} (Chow later proved that the round metric on $\mathbb{S}^{2}$ is a global attractor for the normalised Ricci flow starting at any initial metric \cite{Chow}).\\ 
\\
The methods used in \cite{KS} and \cite{KlKr1} to show a destabilising eigenfunction exists on $\mathbb{CP}^{n}$ use the generalised Hopf fibration to lift the problem to finding certain $U(1)$-invariant functions on the sphere $\mathbb{S}^{2n+1}\subset \mathbb{C}^{n+1}$. This paper takes a totally different approach by viewing the Grassmannians as adjoint orbits of $SU(n)$ and using techniques coming from symplectic geometry (such as the Duistermaat--Heckman formula) to construct eigenfunctions and make calculations of the relevant integrals. \\
\\
\textit{Acknowledgements:} We would like to thank the referees for their careful reading of the paper and their helpful suggestions for improvements.

\section{Background}\label{sec:2}
\subsection{Stability}
Einstein metrics $g$ satisfying $\Ric(g)=\frac{1}{2\tau} g$ for $\tau \in \mathbb{R}$, evolve via homothetic scaling under the Ricci flow
\begin{equation*}
\dfrac{\partial g}{\partial t} =-2\Ric(g).
\end{equation*}	
 It is therefore useful to  view Einstein metrics as fixed points of the Ricci flow up to a normalisation of the volume of the metric by homothetic scaling. A natural question is whether a given Einstein metric is stable as a fixed point in the sense that the Ricci flow starting at a small perturbation of the metric will return to the original Einstein metric.  Perelman \cite{Per1} introduced a functional $\nu(g)$ which is stationary at shrinking gradient Ricci solitons (every Einstein metric with $\tau>0$ is such a soliton) and which is otherwise strictly increasing along the Ricci flow. This allows the stability of an Einstein metric to be investigated by calculation of the second variation of the $\nu$ functional along potentially destabilising directions. If the entropy increases along a particular direction then the corresponding perturbation of the Einstein metric will never return under the flow. This process was first carried out for Einstein metrics by Cao, Hamilton, and Ilmanen \cite{CHI} and later generalised by Cao and Zhu \cite{CZ}.

\begin{theorem2}[Cao--Hamilton--Ilmanen \cite{CHI}]\label{CHIstab}
	Let $(M,g)$ be an Einstein metric with Einstein constant $\frac{1}{2\tau}>0$. Let $h\in s^{2}(T^{\ast}M)$. Then
	$$
	\frac{d^{2}}{ds^{2}}\nu(g+sh) |_{s=0} = \frac{\tau}{\vol(M,g)}\int_{M}\langle N(h),h\rangle dV_{g},	
	$$
	where
	\begin{equation}\label{Stabop}
	N(h) = \frac{1}{2}\Delta h+\Riem(h,\cdot)+\div^{\ast}\div (h)+\frac{1}{2}\nabla^{2}v_{h} -\frac{g}{2n\tau\vol(M,g)}\int_{M}\tr(h)dV_{g},
	\end{equation}
	and $v_{h}$ is the unique solution to
	$$ \Delta v_{h}+\frac{v_{h}}{2\tau}=\div(\div (h)). $$	
\end{theorem2}
The diffeomorphism and scale invariance of $\nu(g)$ means that to check linear stability one only needs to consider perturbations  $h\in s^{2}(T^{\ast}M)$ satisfying 
$$\div(h) = 0 \textrm{ and } \langle h,g\rangle_{L^{2}} = \int_{M}\tr(h)dV_{g} = 0.$$
In this case, the stability operator $N$ in Equation (\ref{Stabop}) reduces to 
$$N(h)  = \frac{1}{2}\Delta h +\Riem(h,\cdot) = \frac{1}{2}(\Delta_{L}+\frac{1}{\tau})h,$$
where  $\Delta_{L}$ is the Lichnerowicz Laplacian 
$$\Delta_{L}h := \Delta h +2\Riem(h,\cdot)-\Ric\cdot h-h\cdot \Ric.$$
We thus have the following definitions:
\begin{definition}[Linear stability of Einstein metrics]
Let $(M,g)$ be a compact Einstein manifold satisfying $\Ric(g) =\dfrac{1}{2\tau} g$ and let $-\kappa$ be the largest eigenvalue of the Lichnerowicz Laplacian restricted to the space of divergence-free, $g$-orthogonal tensors.
\begin{enumerate}
	\item If $\kappa > \dfrac{1}{\tau}$, $g$ is called \textit{linearly stable}. \vspace{5pt}
	\item If $\kappa =\dfrac{1}{\tau}$, $g$ is called \textit{neutrally linearly stable}. \vspace{5pt}
	\item If $\kappa < \dfrac{1}{\tau}$, $g$ is called \textit{linearly unstable}.   
\end{enumerate}
\end{definition}

\begin{definition}[Dynamical stability of Einstein metrics]
Let $(M^{n},g_{E})$ be a compact Einstein manifold. The metric $g_{E}$ is said to be \textit{dynamically stable} for the Ricci flow if for any $m\geq 3$ and any $C^{m}$-neighbourhood $U$ of $g_{E}$ in the space of sections $\Gamma(s^{2}(T^{\ast}M))$, there exists a $C^{m+2}$ neighbourhood of $g_{E}$, $V \subset U$, such that: 
\begin{enumerate}
	\item for any $g_{0}\in V$, the volume normalised Ricci flow 
	$$ \dfrac{\partial g}{\partial t} =-2\Ric(g)+\frac{2}{n}\left(\int_{M}\textrm{scal}(g)dV_{g}\right)g,$$
	with $g(0) = g_{0}$ exists for all time,
	\item the metrics $g(t)$ converge modulo diffeomorphism to an Einstein metric in $U$.
\end{enumerate}
	We call the metric $g_{E}$ {\it dynamically unstable} if there exists a non-trivial normalized Ricci flow defined on $(-\infty,0]$ which converges modulo diffeomorphism to $g_{E}$ as $t\rightarrow -\infty$. 
\end{definition}

The relationship between linear stability and various notions of dynamical stability was pioneered by Sesum \cite{Sesum}. In particular, under the assumption that all infinitesimal Einstein deformations are integrable, Sesum proved that linear stability implies dynamical stability.  In \cite{KKCVP}, Kr\"oncke built on the work of Haslhofer and M\"uller \cite{HasMul} and showed that an Einstein metric with positive Einstein constant is dynamically stable if and only if it is a local maximum of the $\nu$ functional. This characterisation does not require an infinitesimal perturbation to satisfy any integrability assumptions and forms the basis of Kr\"oncke's stability criterion in Theorem \ref{KKstabthm}.  \\
\\
We remark that the definition of dynamical instability requires the existence of a non-trivial ancient flow emerging from the Einstein metric.  This is much stronger than saying that a metric is unstable if it is not dynamically stable. Hence Kr\"oncke's stability theorem in \cite{KKCVP} combined with Theorem \ref{ThmGr} yields the existence of such an ancient flow emerging from the K\"ahler--Einstein metric on the Grassmannians $Gr_{k}(\mathbb{C}^{n})$ (except for the $n=2k$ case).\\
\\
In general, it is very difficult to analyse the spectrum of the Lichnerowicz Laplacian for an arbitrary Einstein metric.  If the metric has some extra structure then more can be said.  In the case the Einstein metric is K\"ahler--Einstein then there is the following topological condition (originally stated in \cite{CHI} and proved for the more general class of K\"ahler--Ricci solitons in \cite{HMPAMS}) 
\begin{theorem2}[Cao--Hamilton--Ilmanen]
Let $(M,J,g)$ be a K\"ahler--Einstein metric. If the Hodge number $h^{1,1}(M)>1$ then $g$ is linearly unstable.
\end{theorem2}
This proposition can be seen as generalising the fact that any product of Einstein metrics with fixed Einstein constant $\frac{1}{2\tau}$ is unstable under the Ricci flow. The product of any two K\"ahler--Einstein metrics always has $h^{1,1}(M)>1$.\\
\\
In \cite{CH}, Cao and He made a complete study of the stability of the simply-connected, compact, irreducible, symmetric spaces. The spaces where the metric is K\"ahler--Einstein can be written in the form $M=G/H$ where $G$ is a connected compact simple Lie group and $H$ is the isotropy subgroup.  We note that the identity component of the isometry group $\textrm{Iso}_{0}=G$ and so the Lie algebra of Killing fields $\mathfrak{k} = Lie(\textrm{Iso}_{0})$  is isomorphic to the Lie algebra $\mathfrak{g}$. All the manifolds in the following theorem have $h^{1,1}(M)=1$.
\begin{theorem2}[Cao--He, c.f. Theorem 4.3 in \cite{CH}] \label{CHT}
The linear stability of the irreducible compact Hermitian symmetric spaces $M ={G}/{H}$ is as follows:
	\begin{enumerate}
		\item $M$ is linearly unstable if:
		\begin{itemize}
			\item $M$ is the space of compatible complex structures on $\mathbb{H}^{n}$,  $M = {Sp(n)}/{U(n)},$			for $n>1$.
		\end{itemize}
		\item $M$ is neutrally linearly stable if:
	\begin{itemize}
		\item $M$ is a complex Grassmannian ${Gr_{k}(\mathbb{C}^{n})}= {SU(n)}/{S(U(k)\times U(n-k))},$ 
		where $n>2$ and ${0<k<n}$,\vspace{3pt}
		\item $M$ is a complex hyperquadric  $Q_{n} = {SO(n + 2)}/{(SO(n) \times SO(2))},$ 
		where $n \geq 4$,
		\item $M$ is a space of orthogonal almost complex structures on $\mathbb{R}^{2n}$,  
		$M = {SO(2n)}/{U(n)},$ 
		for $n>2$,
		\item $M$ is one of the exceptional spaces, 
		${M= {E_{6}}/{(SO(10)\times SO(2))}},$ or \newline
		${M = {E_{7}}/{(E_{6}\times SO(2))}}.$ 
	\end{itemize}	 
\item If $M$ is the sphere $\mathbb{S}^{2} \cong Gr_{1}(\mathbb{C}^{2}) \cong \dfrac{SO(4)}{U(2)}$ then $M$ is dynamically stable and so linearly stable.
\end{enumerate}
\end{theorem2} 
Missing from this list (as it is not irreducible) is the hyperquadric $${Q_{2} = \dfrac{SO(4)}{SO(2)\times SO(2)}  \cong \mathbb{CP}^{1}\times\mathbb{CP}^{1}}.$$ It is unstable as it is a product. The hyperquadric $Q_{3}\cong \dfrac{Sp(2)}{U(2)}$ has $h^{1,1}=1$ but is nevertheless linearly unstable by a result of Gasqui and Goldschmidt \cite{GG}.	\\ 
\\
What Theorem \ref{CHT} shows is that most of the Hermitian symmetric spaces are neutrally linearly stable.  In particular, the complex projective spaces $\mathbb{CP}^{n} = Gr_{1}(\mathbb{C}^{n+1})$ with $n>1$ are all neutrally linearly stable.\\
\\
On any Einstein manifold $(M,g)$ with Einstein constant $\frac{1}{2\tau}$, if there is an eigenfunction $f$ satisfying $\Delta f=-\frac{1}{\tau}f$ then we define the tensor
$$h_{f} := (\Delta f)g-\Hess(f)+\frac{f}{2\tau}g.$$
It can be shown (\cite{CHI}, \cite{CH},  \cite{HMAGAG1}) that $h_{f}$ is divergence free, $L^{2}$-orthogonal to $g$ and satisfies
$$ \Delta_{L}h_{f} = -\frac{1}{\tau}h_{f}.$$
In 2013 Kr\"oncke proved the following stability criterion by computing the third variation of the $\nu$ functional.
\begin{theorem2}[Kr\"oncke, Theorem 1.7  in \cite{KlKr1}] \label{KKstabthm}
Let $(M,g)$ be an Einstein metric with Einstein constant $\frac{1}{2\tau}$ and let $f$ be an eigenfunction of the Laplacian with eigenvalue $-\frac{1}{\tau}$. If the integral
\begin{equation}\label{KKstab}
\int_{M}f^{3}dV_{g} \neq 0,
\end{equation}
then $g$ is dynamically unstable as a fixed point of the Ricci flow and is destabilised by the tensor $h_{f}$.
\end{theorem2} 
Kr\"oncke then constructed a eigenfunction satisfying the condition (\ref{KKstab}) for the spaces $\mathbb{CP}^{n}$ with $n>1$ and proved:
\begin{cor}\label{C1}
The Fubini--Study metrics on $\mathbb{CP}^{n}$, $n>1$ are dynamically unstable as fixed points of the Ricci flow.
\end{cor}
This result was somewhat unexpected as a long-standing conjecture in the field had included $\mathbb{CP}^{2}$ on the list of stable, four-dimensional geometries for the Ricci flow. Theorem \ref{ThmGr} can be seen as a generalisation of the $\mathbb{CP}^{n}$ results of Kr\"oncke and Knopf--Sesum; however, as mentioned in the introduction, our construction of eigenfunctions and method of evaluating the integral is totally different from the methods used in \cite{KS} and \cite{KlKr1}.\\ 
\\
Given Theorem \ref{ThmGr}, it seems likely that the other compact irreducible Hermitian symmetric spaces are also unstable but that they represent metrics with high degrees of degeneracy as critical points of the $\nu$ functional. Further evidence for their instability might come from the behaviour of other non-Hermitian compact symmetric spaces. The first author \cite{Hall19} has investigated the stability of the canonical metric on the compact simple Lie group $G_{2}$  which is also neutrally linearly stable with the neutral directions coming from conformal perturbations corresponding to eigenfunctions of the Laplacian. In this case, Kr\"oncke's stability integral is non-zero for a certain eigenfunction and so $G_{2}$ is unstable.       

\subsection{Geometry of coadjoint orbits}
In this section $G$ is a compact, semisimple Lie group. Henceforth $g$ will denote an element of $G$ (not a Riemannian metric as it has done previously). As $G$ is semisimple, the Killing form $\langle\cdot,\cdot \rangle$ is non-degenerate and so the adjoint representation of $G$ is orthogonal. We can also use the Killing form (or any $\mathrm{Ad}_{G}$ -invariant inner product) to identify $\mathfrak{g}$ and $\mathfrak{g}^{\ast}$. The coadjoint action of $G$ on $\mathfrak{g}^{\ast}$ is defined by
$$\mathrm{Ad}_{g}^{\ast}(\xi)(X) := \xi(\mathrm{Ad}_{g^{-1}}(X))$$
for $g\in G$,  $\xi \in \mathfrak{g}^{\ast}$ and $X \in \mathfrak{g}$.  If ${\xi(\cdot) = \langle\cdot,X\rangle}$ then ${\mathrm{Ad}^{\ast}_{g}(\xi)(\cdot) = \langle \cdot, \mathrm{Ad}_{g}(X)\rangle}$ and we have a straightforward identification of coadjoint and adjoint orbits via the Killing form. \\
\\
For $\xi \in \mathfrak{g}$ we consider the orbit $\mathcal{O}_{\xi}$ of $\xi$ under the adjoint action of $G$. Denote by $H$ stabiliser of $\xi$ and let $\mathfrak{h}\subset \mathfrak{g}$ be its Lie algebra. Then $\mathfrak{g}  =\mathfrak{h} \oplus \mathfrak{m}$ where $\mathfrak{m}$ can be identified with the tangent space to $\mathcal{O}_{\xi}$ at $\xi$. The subalgebra $\mathfrak{h}$ is the kernel of the map ${\mathrm{ad}({\xi}):\mathfrak{g} \rightarrow \mathfrak{g}}$ and thus $\mathfrak{m}$ is the image.\\
\\
We let $T$ be a maximal torus of $G$ and take $\mathfrak{t} = Lie(T)$ to be its Lie algebra. The Weyl group ${W=N_{G}(T)/T}$ where $N_{G}(T)$ is the normaliser of $T$ in $G$. A classical theorem (see for example Bott \cite{Bott}) yields:
\begin{enumerate}
	\item $\mathcal{O}_{\xi}\cap \mathfrak{t} \neq \emptyset$,
	\item $\mathcal{O}_{\xi}\cap \mathfrak{t}$ is a $W$-orbit.
\end{enumerate}
This means, without loss of generality, we can take the element representing the orbit $\xi \in \mathfrak{t}$.\\
\\
The orbits have the structure of a complex manifold. Decomposing the complexified Lie algebra $\mathfrak{g} \otimes \mathbb{C}$ we get
$$ \mathfrak{g} \otimes \mathbb{C} =  \mathfrak{t}_{\mathbb{C}} \oplus\left( \bigoplus_{\alpha: \ \langle \alpha, \xi\rangle=0} R_{\alpha} \right)\oplus A\oplus \bar{A}$$
where $\alpha \in \mathfrak{t}_{\mathbb{C}}$ are the roots of $G$, $R_{\alpha}$ is the root space of $\alpha$, and $A$ is the span of the root spaces satisfying
$$[\xi, r_{\alpha}] =i\langle \alpha, \xi\rangle r_{\alpha},  \textrm{ with } \langle \alpha, \xi\rangle >0 \textrm{ for all } r_{\alpha}\in R_{\alpha}.$$
One can identify $\mathfrak{m}_{\mathbb{C}}\cong A\oplus \bar{A}$ and show that  $A$ and $\mathfrak{t}_{\mathbb{C}} \oplus A$ are Lie subalgebras of $ \mathfrak{g} \otimes \mathbb{C}$. By defining $\mathfrak{m}^{(1,0)}=A$ we get a $G$-invariant complex structure on $\mathcal{O}_{\xi}$.\\
\\
The Kirillov--Kostant--Souriau symplectic form is defined as
$$\omega_{\xi}(x,y) = -\langle\xi,[x,y]\rangle,$$
for $x,y \in \mathfrak{m}$. This is extended over  $\mathcal{O}_{\xi}$ using the adjoint action. This form is compatible with the complex structure and gives the orbit the structure of a K\"ahler manifold. In the case when the center of $H$ has dimension 1 the induced metric is K\"ahler--Einstein (c.f.\cite{Bes} Proposition 8.85). This always holds for all  manifolds considered in this paper; in fact the K\"ahler--Einstein metric is precisely the Hermitian symmetric space metric.  
\subsection{Properties of the eigenfunctions}
We will now show how to construct eigenfunctions for the Laplacian of the K\"ahler--Einstein metric on $\mathcal{O}_{\xi}$. We begin by defining functions ${f_{\eta} \in C^{\infty}(\mathcal{O}_{\xi})}$ by
\begin{equation} \label{ef_def}
f_{\eta}(Z): =  \langle Z,\eta \rangle,
\end{equation}
where $Z\in \mathcal{O}_{\xi}$ and $\langle \cdot, \cdot \rangle$ is the inner product on $\mathfrak{g}$ coming from the Killing form. These functions satisfy some important properties.
\begin{lemma} \label{f_mv_lem}
Suppose that $\eta, \tilde{\eta} \in \mathfrak{g}$ are in the same $G$-orbit. Then the functions $f_{\eta},f_{\tilde{\eta}} \in C^{\infty}(\mathcal{O}_{\xi})$ defined by Equation (\ref{ef_def}) satisfy 
$$
\int_{\mathcal{O}_{\xi}}f_{\eta}^{k} \ \omega^{n}  = \int_{\mathcal{O}_{\xi}}f_{\tilde{\eta}}^{k} \ \omega^{n},  
$$ 
where $k\in \mathbb{N}$, $n$ is the complex dimension of $\mathcal{O}_{\xi}$, $\omega$ is the Kirillov--Kostant--Souriau symplectic form, and where $\omega^{n}$ the volume form of the associated K\"ahler--Einstein metric. Furthermore, in the case that $k=1$ we have
$$\int_{\mathcal{O}_{\xi}} f_{\eta} \ \omega^{n} = 0,$$
for all $\eta \in \mathfrak{g}$.
\end{lemma}
\begin{proof}
The conditions of the lemma mean there is a $g \in G$ such that $\tilde{\eta} = \mathrm{Ad}_{g}(\eta)$. Hence by the Ad-invariance of the inner product we have
$$ f_{\eta}(Z) = \langle \eta,Z\rangle  = \langle \mathrm{Ad}_{g^{-1}}(\tilde{\eta}),Z\rangle  = \langle \tilde{\eta}, \mathrm{Ad}_{g}(Z)\rangle = f_{\tilde{\eta}}(\mathrm{Ad}_{g}(Z)).$$
In other words, ${f_{\eta} = \mathrm{Ad}_{g}^{\ast}(f_{\tilde{\eta}})}$. As ${\mathrm{Ad}_{g}:\mathcal{O}_{\xi}\rightarrow\mathcal{O}_{\xi}}$ is an orientation preserving isometry, we have
$$
\int_{\mathcal{O}_{\xi}}f_{\eta}^{k} \ \omega^{n}=  \int_{\mathcal{O}_{\xi}}\mathrm{Ad}_{g}^{\ast}(f_{\tilde{\eta}})^{k} \ \omega^{n} = \int_{\mathcal{O}_{\xi}}\mathrm{Ad}^{\ast}_{g}(f_{\tilde{\eta}})^{k} \ (\mathrm{Ad}_{g}^{\ast}\omega^{n}) =  \int_{\mathcal{O}_{\xi}}f_{\tilde{\eta}}^{k} \ \omega^{n}.
$$
To prove the second part of the lemma, we note that the function ${F: G\rightarrow \mathbb{R}}$ given by
$$F(g) = \int_{\mathcal{O}_{\xi}}f_{\mathrm{Ad}_{g}(\eta)} \omega^{n},$$
is constant. Taking the derivative at the identity yields 
$$\int_{\mathcal{O}_{\xi}}f_{[\eta,\zeta]}\omega^{n} = 0,$$
for all $\eta, \zeta \in \mathfrak{g}$. This means that the map
$$\eta \rightarrow \int_{\mathcal{O}_{\xi}}f_{\eta} \ \omega^{n}$$ is a Lie algebra homomorphism. The fact that $G$ is simple means that this map must be the trivial homomorphism and so the result follows. 
\end{proof}
 Next we recall  a theorem of Matsushima \cite{Mat} which says that for any Fano K\"ahler--Einstein manifold $(M,g,J)$ there is an isomorphism between the $(-1/\tau)$-eigenspace, $E_{(-1/\tau)}$, and the Lie algebra of Killing vector fields $\mathfrak{k}$  given by
$$\phi \rightarrow -J\nabla\phi,$$
where $J$ is the complex structure.\\
\\
All the connected, compact, irreducible Riemannian symmetric spaces  can be constructed in the form $M=\textrm{Iso}_{0}/\textrm{Iso}_{p}$ where $\textrm{Iso}_{0}$ is the connected component of the identity of the isometry group of $(M,g)$ and $\textrm{Iso}_{p}$ is the isotropy group of isometries fixing a point. For the spaces $G/H$ in Theorem \ref{CHT} we have
$$\mathfrak{g}\cong \textrm{Lie}(\textrm{Iso})\cong \mathfrak{k}.$$
This map can be realised by the assignment
$$\eta\rightarrow \frac{d}{dt} \bigg|_{t=0} \mathrm{Ad}_{\exp(t\eta)}(Z) = [\eta,Z],$$
for $\eta \in \mathfrak{g}$ and $Z \in \mathcal{O}_{\xi}$. 
\begin{lemma}\label{eflem}
If the K\"ahler--Einstein metric $g_{KE}$ on $\mathcal{O}_{\xi}$ has Einstein constant $\frac{1}{2\tau}$, then the functions $f_{\eta}$ defined in Equation (\ref{ef_def}), satisfy
$$\Delta f_{\eta} = -\frac{1}{\tau}f_{\eta}.$$	
\end{lemma}
\begin{proof}
	Let $X \in \mathfrak{m}$ and consider 
	$$F(t)=f_{\eta}(\mathrm{Ad}_{\exp(tX)}\xi) = \langle  \eta, \mathrm{Ad}_{\exp(tX)}\xi \rangle = \langle \mathrm{Ad}_{\exp(-tX)}\eta, \xi \rangle. $$
	Taking derivatives we see
	$$\dfrac{dF}{dt}\Bigr\vert_{t=0} =-\langle \xi, [X,\eta]\rangle = -\omega(X,\eta) = g_{KE}(X,J\eta).$$
	Hence $\nabla f_{\eta} =J\eta$ (here we identify $\eta$ with the Killing field it generates on $\mathcal{O}_{\xi}$).  As $\eta$ is a Killing field we invoke Matsushima's theorem  which says the map
	$$ \phi \rightarrow \nabla \phi,$$
	is an isomorphism between the eigenspace $E_{-1/\tau}$ and $J\mathfrak{k}$ where $\mathfrak{k}$ is the space of Killing fields. Hence, as $f_{\eta}$ has mean value zero by Lemma \ref{f_mv_lem}, we see $f_{\eta}$ is an eigenfunction of the Laplacian with eigenfunction $-\frac{1}{\tau}$.
\end{proof}
An element $X\in \mathfrak{g}$ is said to be $\textit{regular}$ if its centraliser in $\mathfrak{g}$ is of smallest possible dimension. The regular elements of $\mathfrak{t}$ we will denote by $\mathfrak{t}_{\textrm{reg}}$. The following lemma will be vital in our analysis.
\begin{lemma}[Bott \cite{Bott}]\label{Bottlem}
Let $D \in \mathfrak{t}_{\mathrm{reg}}$ be regular. Then the function $f_{D}$ defined by Equation (\ref{ef_def}) has the following properties.
\begin{enumerate} 
 \item  The critical points of $f_{D}$ are non-degenerate and of even index.
 \item  The critical points are the orbit of $\xi \in \mathfrak{t}$ under the action of Weyl group $W$.
\end{enumerate}	
\end{lemma}
We remark that $f_{D}$ is a Hamiltonian function for the  action on $\mathcal{O}_{\xi}$ generated by $D$. If $\Lambda \subset \mathfrak{t}$ is the weight lattice, then choosing $D\in \Lambda\otimes_{\mathbb{Z}}\mathbb{Q}$ will generate an $\mathbb{S}^{1}$-action. We denote the set of elements in $\mathfrak{t}$ that generate closed orbits by $\mathfrak{t}_{c}$.  The set $\mathfrak{t}_{reg}\cap \mathfrak{t}_{c}$ is dense in $\mathfrak{t}$ (with respect to the Euclidean topology). It turns out that the eigenfunctions $f_{D}$ generated by  $D \in \mathfrak{t}_{reg}\cap \mathfrak{t}_{c}$  are the only ones that one needs to check the stability condition (\ref{KKstab}) on.
\begin{proposition}\label{Proptorstab}
Let $G/H=\mathcal{O}_{\xi}$ be one of the symmetric spaces in Theorem \ref{CHT} and let $T$ be a maximal torus in $G$ with Lie algebra $\mathfrak{t}$. Suppose that there exists $f \in E_{-(1/\tau)}$ such that 
$$\int_{\mathcal{O}_{\xi}}f^{3}\omega^{n} \neq0,$$
then there exists $D\in \mathfrak{t}_{reg}\cap \mathfrak{t}_{c}$ such that
$$\int_{\mathcal{O}_{\xi}}f_{D}^{3} \ \omega^{n} \neq0.$$
\end{proposition}
\begin{proof}
As the map $\varrho:\mathfrak{g}\rightarrow E_{-1/\tau}$ given by $\varrho(X) = f_{X}$ is an isomorphism, there must exist $\eta\in \mathfrak{g}$ such that $f=f_{\eta}$. As remarked in Section 2.2, given a fixed maximal torus $T$ with Lie algebra $\mathfrak{t}$, any element of $\mathfrak{g}$ is in the adjoint orbit of some element in $\mathfrak{t}$. If $D = \mathrm{Ad}_{g}(\eta)\in \mathfrak{t}$ for some $g \in G$ then, by Lemma \ref{f_mv_lem}
$$ \int_{\mathcal{O}_{\xi}}f_{D}^{3} \ \omega^{n} \neq 0.$$   
Finally we see we can take $D\in \mathfrak{t}_{reg}\cap \mathfrak{t}_{c}$ as this set is dense in $\mathfrak{t}$.
\end{proof}
\subsection{Computing integrals}
We will need to be able to compute integrals of powers of $f$ over the orbit.  This is achieved by the famous Duistermaat--Heckman formula \cite{DH} (see \cite{MS} for the form we are using). On a symplectic manifold $(M^{2n},\omega)$ with a Hamiltonian circle action that has an associated Hamiltonian function $\varphi$ with non-degenerate critical points, the  Duistermaat--Heckman formula  is
$$\int_{M}e^{-t\varphi}\omega^{n}= \frac{n!}{t^{n}}\sum_{q \textrm{ critical }}\frac{e^{-t\varphi(q)}}{\varpi(q)},$$
where the $\varpi(q)$ is the product of the weights of the circle action that is induced on the tangent space at each fixed point $q$.\\
\\
The holomorphic tangent space at a critical point $q$ is identified with the span of the root spaces $R_{\alpha}$ satisfying ${\langle q, \alpha\rangle>0}$, we will denote this set $P(q)$. The derivative of ${r_{\alpha} \rightarrow ad_{\exp(tD)}r_{\alpha}}$ at $t=0$  is $\langle \alpha,D \rangle r_{\alpha}$. Hence the weight at each fixed point $q \in \mathfrak{t}$ is given by
$$\varpi(q) = \prod_{\alpha \in P(q)}\langle \alpha, D\rangle.$$
\begin{proposition}
Let $D\in \mathfrak{t}_{reg}\cap\mathfrak{t}_{c}$ and let ${f_{D}(Z) = \langle Z, D\rangle}.$ Then
\begin{equation}\label{DHGroup}
\int_{\mathcal{O}_{\xi}}e^{-tf}\omega^{n} = \frac{n!}{t^{n}}\sum_{w\in W \backslash \textrm{stab}(\xi)}\frac{e^{-tf(w\cdot \xi)}}{\prod_{\alpha \in P(w\cdot \xi)}\langle \alpha, D\rangle},
\end{equation}
where $n$ is the complex dimension of the orbit $\mathcal{O}_{\xi}$.
\end{proposition}
\begin{proof}
We simply apply the Duistermaat--Heckman formula to the function $f_{D}$ which is the Hamiltonian for the circle action generated by $D\in \mathfrak{t}_{reg}\cap\mathfrak{t}_{c}$. As the element $D$ is regular, Lemma \ref{Bottlem} says the critical points are non-degenerate and precisely the orbit of $\xi$ under the action of the Weyl group $W$. The value of the weights follows from the previous discussion.  
\end{proof}
Formulae similar to (\ref{DHGroup}) occur in the theory of the orbit method developed by Kirillov \cite{Kir}. In this theory integrals over certain coadjoint orbits correspond to the characters of the representation corresponding to the orbit. Similar expressions also appear elsewhere in the literature, for example in the papers of Berline and Vergne \cite{BV}, Paradan \cite{Par}, and Rossman \cite{Ros}.
\section{The proof of Theorem \ref{ThmCST}} \label{sec:3}

The proof of Theorem \ref{ThmCST} is based on the classification of certain invariant polynomial algebras.  If we let $G$ be one of the compact simple Lie groups appearing as the larger group in Theorem \ref{CHT}, then we denote by $\mathbb{R}[\mathfrak{g}]^{G}$ the graded algebra of $G$-invariant polynomials functions on $\mathfrak{g}$. If $T$ is a maximal torus of $G$ with Lie algebra $\mathfrak{t}$ and Weyl group $W$, then the Chevalley restriction theorem (see for example Section 23 in \cite{Hum}) yields an isomorphism ${\mathcal{I}:\mathbb{R}[\mathfrak{g}]^{G}\rightarrow \mathbb{R}[\mathfrak{t}]^{W}}$ where $\mathcal{I}(P)$ is simply restriction of a polynomial $P\in \mathbb{R}[\mathfrak{g}]^{G}$ to $\mathfrak{t}$. By the Chevalley--Shephard--Todd theorem and the Shephard--Todd classification of complex reflection groups, see \cite{Coxe} and \cite{ST}, the algebra ${\mathbb{R}[\mathfrak{t}]^{W}}$  is a polynomial algebra, with generators of well defined degrees. We list them in the following table with the groups in brackets having Lie algebra with the corresponding root system:
\begin{center}
\begin{tabular}{ |c|c| } 
 \hline
 Root system  & Degree of generators  in ${\mathbb{R}[\mathfrak{t}]^{W}}$  \\
 \hline 
 $A_{n} \ (SU(n+1))$ &  $2,3,\dots, n+1$  \\ 
 $B_{n} \ (SO(2n+1))$ &  $2,4, \dots, 2n$\\ 
 $C_{n} \ (Sp(n))$ & $2,4, \dots, 2n$\\
 $D_{n} \ (SO(2n))$ & $2,4,\dots, 2(n-1); n$ \\
 $E_{6}$ & $2,5,6,8,9,12$\\
 $E_{7}$ & $2,6,8,10,12,14,18$\\
 \hline
\end{tabular}
\end{center}
\vspace{10pt}
It follows from the first entry in the table above that the space of degree 3, $SU(n)$ invariant polynomial functions on $\mathfrak{su}(n)$ is one dimensional. Indeed, the space of degree 0 elements is spanned by the constant polynomial 1, and the multiplication of any pair of generators of positive degree is of degree at least 4. (Note that the results in \cite{Coxe,Hum,ST} cited above concern semi-simple complex Lie groups and complex reflection groups. However, the complexification $G_{\mathbb{C}}$ of a simple compact Lie group $G$ is semi-simple, and $\mathbb{C} \otimes_{\mathbb{R}} \mathbb{R}[\mathfrak{g}]^{G} \cong \mathbb{C}[\mathfrak{g}_{\mathbb{C}}]^{G_{\mathbb{C}}}$ as graded algebras where $\mathfrak{g}_{\mathbb{C}}$ is the Lie algebra of $G_{\mathbb{C}}$. In particular, $\mathbb{R}[\mathfrak{g}]^{G}$ has a set of generators of the same degrees as that of $\mathbb{C}[\mathfrak{g}_{\mathbb{C}}]^{G_{\mathbb{C}}}$.)

The following lemma shows that the stability integral (\ref{KKstab}) can be thought of as an element of ${\mathbb{R}[\mathfrak{g}]^{G}}$.
\begin{lemma}\label{homflem}
Let $I^{k}:\mathfrak{g}\rightarrow\mathbb{R}$ be defined by
$$I^{k}(\eta) = \int_{\mathcal{O_{\xi}}}(f_{\eta})^{k}\omega^{n}.$$
Then $I^{k}$ is a (possibly trivial) $\mathrm{Ad}_{G}$-invariant, homogenous, degree$-k$ polynomial.
\end{lemma}
\begin{proof}
The $\mathrm{Ad}_{G}$-invariance was demonstrated in Lemma \ref{f_mv_lem}. The fact that the function is a homogenous, degree-$k$ polynomial is straightforward if one picks a basis of $\mathfrak{g}$ and then calculates in coordinates.
\end{proof}
We now give the proof of Theorem \ref{ThmCST}. 
\begin{proof}
We note by Proposition \ref{Proptorstab}, we might as well assume that a destabilising eigenfunction is of the form $f_{D}$ for $D \in \mathfrak{t}_{reg}\cap\mathfrak{t}_{c}$. If the symmetric space is not a Grassmanian then it is of the form $G/H$ with $G$ being one of the groups $B_{n},D_{n},E_{6}$ or $E_{7}$. Lemma  \ref{homflem} shows that $I^{3}$ is a degree $3$, homogenous $G$-invariant polynomial and so by the Chevalley restriction theorem yields a degree 3 element of  ${\mathbb{R}[\mathfrak{t}]^{W}}$.  However, the above table shows this function must vanish unless $G=D_{3} = SO(6) =A_{3}$. One can show that $SO(6)/U(3) \cong \mathbb{CP}^{3}$ and  $Q_{4}  \cong Gr(2,4)$  and so the stability of these spaces follows from the type $A_{n}$ consideration.
\end{proof}

We also note that this method of proof also shows that the functions $f_{\eta}$ defined by Equation (\ref{ef_def}) have mean value $0$ (which was demonstrated explicitly in Lemma \ref{f_mv_lem}). This follows as there are no non-zero homogeneous, degree $1$ polynomials that are invariant under the Weyl groups of the compact connected Lie groups we are considering in this article. 
\section{Combinatorial properties of certain determinants} \label{sec:4}
\subsection{General Case}
In order to prove Theorem \ref{ThmGr} we first need to collect some results on the power series of certain matrices that will appear after the manipulation of the righthand side of Equation (\ref{DHGroup}).  For ${m_{1},m_{2},\dots,m_{n} \in \mathbb{R}}$ and ${0<k \leq [n/2]}$ we consider the matrix valued function ${\mathcal{M}:\mathbb{R}\rightarrow \mathrm{Mat}^{n\times n}_{\mathbb{R}}}$ given by 
\begin{equation}\label{DHMatrix}
\mathcal{M}(t) = \left(
\begin{array}{cccc}
e^{-m_{1}t} & e^{-m_{2}t} & \dots  & e^{-m_{n}t} \\ 
e^{-m_{1}t}m_{1} & e^{-m_{2}t}m_{2} & \dots  & e^{-m_{n}t}m_{n}\\
\vdots &\vdots &\vdots &\vdots\\
e^{-m_{1}t}m_{1}^{k-1} & e^{-m_{2}t}m_{2}^{k-1} & \dots  & e^{-m_{n}t}m_{n}^{k-1}\\ 
1& 1 & \dots  & 1 \\ 
m_{1} & m_{2} & \dots  & m_{n} \\ 
\vdots &\vdots &\vdots &\vdots\\
m_{1}^{n-k-1} & m_{2} ^{n-k-1}& \dots  & m_{n}^{n-k-1} 
\end{array}
\right).
\end{equation}
 We will be interested in computing the derivatives of $\det(\mathcal{M}(t))$ when $t=0$; it is therefore useful to think of $\det(\mathcal{M}(t))$ as the sum of various products of $k$ functions. To compute the $p^{th}$ derivative we can use the multinomial version of the Leibniz rule   
$$\dfrac{d^{p}}{dt^{p}}\det(\mathcal{M}(t)) = \sum_{d_{1}+d_{2}+\dots+d_{n}=p}\left(\dfrac{p!}{ d_{1}! d_{2}!\dots d_{n}!}\right)\det(\mathcal{M}(R_{1}^{(d_{1})}, R_{2}^{(d_{2})},\dots,R_{n}^{(d_{n})})), $$
where $d_{i}\in \mathbb{N}\cup\{0\}$ and $\mathcal{M}(R_{1}^{(d_{1})}, R_{2}^{(d_{2})},\dots,R_{n}^{(d_{n})})$ is the matrix formed by taking $d_{i}$ derivatives of the terms in the $i^{th}$ row (we note whenever a non-zero derivative is applied to the final $n-k$ rows the term in the sum will vanish).\\
\\
It is clear that, evaluating at $t=0$, this formula is going to require the calculation of determinants of matrices of the form
$$
A =\left(
\begin{array}{cccc}

m_{1}^{e_{1}} & m_{2}^{e_{1}} & \dots & m_{n}^{e_{1}}\\ 
m_{1}^{e_{2}} & m_{2}^{e_{2}} & \dots & m_{n}^{e_{2}}\\ 
\vdots &\vdots &\vdots &\vdots\\
m_{1}^{e_{k}} & m_{2}^{e_{k}} & \dots & m_{n}^{e_{k}}\\
m_{1}^{n-k-1} & m_{2} ^{n-k-1}& \dots  & m_{n}^{n-k-1}\\
m_{1}^{n-k-2} & m_{2}^{n-k-2} & \dots  & m_{n}^{n-k-2} \\ 
\vdots &\vdots &\vdots &\vdots\\
m_{1} & m_{2} & \dots & m_{n}\\
1 & 1 & \dots & 1
\end{array}
\right),
$$
for exponents ${e_{1},e_{2},\dots,e_{k}\in \mathbb{N}}$. In fact, a matrix of the form $A$ that gives a non-zero determinant can be written (after possibly reordering rows) in the form
\begin{equation}\label{lambdamat}
A_{ij} = m_{j}^{\lambda_{i}+n-i},
\end{equation}
for some vector $\lambda = (\lambda_{1}, \lambda_{2},\dots, \lambda_{n}) \in\mathbb{Z}_{\geq 0}^{n}$ with $\lambda_{i}\geq \lambda_{i+1}$ and with $\lambda_{i}=0$ for $i>k$. Such determinants are all multiples of the Vandermonde determinant $V$ given by
$$ V = |m_{j}^{n-i}| = \prod_{1\leq i<j\leq n}(m_{i}-m_{j}),$$
which is the determinant of the $\lambda=(0,0,\dots,0)$ case of Equation (\ref{lambdamat}). The quotients formed this way turn out to be part of a very well-known set of functions known as the Schur polynomials.
\begin{definition}[Schur polynomial]
	Given $\lambda = (\lambda_{1}, \lambda_{2},\dots, \lambda_{n}) \in\mathbb{Z}_{\geq 0}^{n}$ with $\lambda_{i}\geq \lambda_{i+1}$ the Schur polynomial $S_{\lambda}$ is given by
	$$S_{\lambda}(m_{1},m_{2},\dots,m_{n}) = \frac{|m_{j}^{\lambda_{i}+n-i}|}{V}.$$
\end{definition}
The Schur polynomial $S_{\lambda}$ is a homogeneous, $\textrm{Sym}_{n}$-invariant multinomial of degree $\sum_{i=1}^{n}\lambda_{i}$.  It is straightforward to write a short list of these in degrees 0 to 3:
\begin{align*}
S_{(0,0,\dots,0)} &= 1,\\
S_{(1,0,\dots,0)} &= m_{1}+m_{2}+\dots+m_{n},\\
S_{(2,0,\dots,0)} &= \sum_{i=1}^{n}m_{i}^{2}+\sum_{1\leq i<j\leq n}m_{i}m_{j}, \\
S_{(1,1,0,\dots,0)} &= \sum_{1\leq i<j\leq n}m_{i}m_{j}, \\
S_{(3,0,0,\dots,0)} &=\sum_{i=1}^{n}m_{i}^{3}+\sum_{1\leq i\neq j\leq n}m_{i}m_{j}^{2}+\sum_{1\leq i<j<l\leq n}m_{i}m_{j}m_{l}, \\
S_{(2,1,0,\dots,0)} &= \sum_{1\leq i\neq j\leq n}m_{i}m_{j}^{2}+2\sum_{1\leq i<j<l\leq n}m_{i}m_{j}m_{l},\\
S_{(1,1,1,0,\dots,0)} &= \sum_{1\leq i<j<l\leq n}m_{i}m_{j}m_{l}.
\end{align*}

The set of  $\textrm{Sym}_{n}$-invariant multinomials in $m_{1}, m_{2}, \dots, m_{n}$ of a fixed degree form a vector space, and the Schur polynomials form a basis for this space (see for example Appendix A of \cite{FH}).\\  
\\
In order to keep track of signs when reordering the rows of the relevant matrices, we note the following lemma, the proof of which we leave to the reader. 
\begin{lemma}\label{reordlem}
Let $\sigma\in \mathrm{Sym}_{N}$ be the permutation 
$$
\sigma = \left(\begin{array}{cccccc}
1 & 2 & 3 &\dots & N-1 & N\\
N & N-1 & N-2 & \dots & 2 & 1
\end{array}\right).
$$
Then $\mathrm{sgn}(\sigma)= (-1)^{\tau(N)}$, where
\end{lemma}
\begin{equation}\label{reordlemeq}
\tau(N) = \left\{\begin{array}{c}0 \ \mathrm{ if } \ N \equiv \ 0,1 \ (\mathrm{mod } \ 4),\\
1  \ \mathrm{ if } \ N\equiv \ 2,3 \ (\mathrm{mod } \ 4).
\end{array} \right.
\end{equation}
\begin{lemma}\label{AL1}
Let $k,n,$ and $\mathcal{M}(t)$ be as in Equation (\ref{DHMatrix}) and let $V$ be the Vandermonde determinant.  Then the power series expansion about $0$ of $\det(\mathcal{M}(t))$ begins
	$${\det(\mathcal{M}(t))} =  \dfrac{\varepsilon_{k,n}Vc_{0}}{(k(n-k))!} \ t^{k(n-k)}+\dots,$$ 
	where
	$$
	\varepsilon_{k,n} =(-1)^{k(n-k)+\tau(k)+\tau(n-k)}, 
	$$
	and
	\begin{equation}\label{deggrass}
	c_{0} = (k(n-k))!\prod_{i=1}^{k}\dfrac{(i-1)!}{(n-k+i-1)!}.
	\end{equation}
\end{lemma}

\begin{proof}
	In order to get a non-zero determinant when $t=0$, we need the first $k$ rows to yield the powers ${m_{j}^{n-k}, m_{j}^{n-k+1},\dots,m_{j}^{n-1}}.$ This means putting at least $(n-k)$ derivatives onto each of the first $k$ rows and thus the first possible non-zero derivative is the $(k(n-k))^{th}$ one. We proceed by computing
$$
\dfrac{d^{k(n-k)}}{dt^{k(n-k)}}\det(\mathcal{M}(t))\bigg|_{t=0}. 	
$$	
For $1\leq i \leq k$ let $l_{i}=n-i$ (we can think of the $l_{i}$ as the remaining powers of the $m_{j}$ we require in order that the resulting matrix is non-singular). For each $\sigma \in \textrm{Sym}_{k}$, let the number of derivatives of the $\sigma(i)^{th}$ row be given by
$$(l_{i}-\sigma(i)+1).$$ 
(Hence the power of $m$ contributed by the $\sigma(i)^{th}$ row is $l_{i}$). 
Using the multinomial version of the Leibniz rule and weighting the resulting matrix by the sign of  the permutation $\sigma$ (as well as reordering the rows so the powers of $m_{j}$ run from $1$ to $n-1$) we obtain the $(k(n-k))^{th}$ derivative of $\det(\mathcal{M}(t))$ at $t=0$ is given by
	$$\varepsilon_{k,n}V\left(\sum_{\sigma \in \textrm{Sym}_{k}} \textrm{sgn}(\sigma)\dfrac{[k(n-k)]!}{(l_{1}-\sigma(k)+1)!(l_{2}-\sigma(k-1)+1)!\dots(l_{k}-\sigma(1)+1)!}\right).$$
	The result now follows from the discussion in \cite{FH} where the quantity inside the brackets is shown to compute the number of standard Young tableau of row structure 
	$${(n-k,n-k,\dots,n-k)\in \mathbb{N}^{k}}.$$
	The formula for $c_{0}$ can be computed by the famous Hook Length formula (see Section 4.1 in \cite{FH}).  This gives the result.
\end{proof}
The formula for $c_{0}$ (up to factors of $\pi$) recovers the volume of $Gr_{k}(\mathbb{C}^{n})$ as first computed by Schubert \cite{GrHa}. We will see that $c_{0}$ is essentially the first term in the expansion of the Duistermaat--Heckman integral (\ref{DHGroup}) which indeed should be the volume of the manifold computed with respect to the symplectic form. \\
\\
In order to compute the stability integral (\ref{KKstab}), we will require the coefficient of $t^{k(n-k)+3}$ in the power series expansion about $0$ of $\det(\mathcal{M}(t))$.  After factoring out $V$, this coefficient will be a combination of the Schur polynomials $S_{(3,0,\dots,0)}$, $S_{(2,1,0\dots,0)}$, and $S_{(1,1,1,0\dots,0)}$. The coefficient of each polynomial can be computed in terms constant $c_{0}$ given by Equation (\ref{deggrass}).
\begin{lemma}\label{AL2}
	Let $\dfrac{c_{3}}{(k(n-k)+3)!}$ be the coefficient of $t^{k(n-k)+3}$ in the power series expansion about $0$ of $\det(\mathcal{M}(t))$ and let $\varepsilon_{k,n}$ be as in Lemma \ref{AL1}. Then
	$$c_{3} = \varepsilon_{k,n}V(c_{(3,0,0)}S_{(3,0,\dots,0)}+c_{(2,1,0)}S_{(2,1,0\dots,0)}+c_{(1,1,1)}S_{(1,1,1,0\dots,0)}),$$
	where		
	\begin{eqnarray}
	c_{(3,0,0)} & = \left(\prod_{i=1}^{3}(k(n-k)+i) \right)\dfrac{k(k+1)(k+2)}{6n(n+1)(n+2)}c_{0},\\
	c_{(2,1,0)} & = \left(\prod_{i=1}^{3}(k(n-k)+i) \right)\dfrac{(k-1)k(k+1)}{3(n-1)n(n+1)}c_{0},\\
	c_{(1,1,1)} & = \left(\prod_{i=1}^{3}(k(n-k)+i) \right)\dfrac{(k-2)(k-1)k}{6(n-2)(n-1)n}c_{0}. 
	\end{eqnarray}
	
\end{lemma}	
\begin{proof}
	Let ${\chi_{1} = (3,0,\dots,0)},$ ${\chi_{2} = (2,1,0,\dots,0)},$ and ${\chi_{3} = (1,1,1,0,\dots,0) \in \mathbb{Z}^{n}}$. For ${i=1,2}$ or $3$, let 
	$${\upsilon = (\underbrace{n-k,n-k,\dots,n-k}_{k \ \mathrm{terms}},0,\dots,0)+\chi_{i}},$$ and denote by $\upsilon_{j}$ the $j^{th}$ entry of $\upsilon$.
	Furthermore, for ${j=1,2,\dots,k}$ let
	$$l_{j} = (\upsilon_{j}+k-j).$$
As in the proof of Lemma \ref{AL1}, the $l_{i}$ are the remaining powers of $m$s we need to obtain to get a non-singular matrix. Let row $\sigma(j)$ have $(l_{j}-\sigma(j)+1)$ derivatives applied to it so, as in the proof of Lemma \ref{AL1}, the power of the $m$s in row $\sigma(j)$ is $l_{j}$. For a fixed element $\sigma \in \textrm{Sym}_{k}$, computing the derivative of $\det(\mathcal{M}(t))$ at $0$ and after dividing through by $V$, the distribution of the powers shows we get the Schur polynomial $S_{\chi_{i}}$.\\ 
	\\
	Using the multinomial version of the Leibniz rule where we sum only over the distribution of derivatives that yield powers of $m$ running  ${1,2,\dots,n+2}$ and weighting the resulting matrix by the sign of the permutation $\sigma$ we obtain that this part of the $(k(n-k)+3)^{rd}$ derivative of $\det(\mathcal{M}(t))/V$ at $t=0$ (and hence the coefficient of $S_{\chi_{i}}$) is given by
	$$\varepsilon_{k,n}V\left(\sum_{\sigma \in \textrm{Sym}_{k}} \textrm{sgn}(\sigma)\dfrac{[k(n-k)+3]!}{(l_{1}-\sigma(k)+1)!(l_{2}-\sigma(k-1)+1)!\dots(l_{k}-\sigma(1)+1)!}\right).$$
	Again, the discussion in \cite{FH} shows that the quantity inside the brackets computes the number of standard Young tableau of row structure $$\left((n-k,n-k,\dots,n-k)+\tilde{\chi}_{i}\right) \in \mathbb{N}^{k},$$ 
where $\tilde{\chi_{i}}$ is the vector formed from the first $k$ entries of $\chi_{i}$.	
The result follows from the Hook Length formula.
\end{proof}

\subsection{Restriction to $\sum_{j=1}^{n} m_{j}=0$}

The results of the previous section are valid for a general set of inputs $m\in \mathbb{R}^{n}$.  In Section \ref{sec:5} we will want to view the vector $m$ as an element of the Lie algebra $\mathfrak{t}\subset \mathfrak{su}(n)$ corresponding to the maximal torus defined by taking the diagonal matrices.  As the matrices in question are trace free, we wish to consider restricting the inputs $m$ to the subspace of $\mathbb{R}^{n}$ defined by $\sum_{j=1}^{n} m_{j}=0$. We will denote this subspace by $\mathcal{T}_{n}$.\\
\\
As remarked previously, in general, the Schur polynomials of a fixed degree are linearly independent; however, when restricted to the subspace  $\mathcal{T}_{n}$, this is no longer the case.  Fortunately, in the degree three case, the Chevalley--Shephard--Todd theorem shows that the restriction is one dimensional and we can compute each restriction in terms of a fixed $\mathrm{Sym}_{n}$-invariant cubic polynomial which for simplicity we choose to be 
$$\sum_{j=1}^{n}m_{j}^{3}.$$

\begin{lemma}\label{ResLem1}
The restriction of the Schur polynomials $S_{(3,0,0,\dots,0)}$, $S_{(2,1,0,\dots,0)}$, and  
$S_{(1,1,1,0\dots,0)} $ to the subspace of $\mathcal{T}_{n}\subset \mathbb{R}^{n}$ yields:

\begin{align}
S_{(3,0,0,\dots,0)} = \dfrac{1}{3}\left(\sum_{j=1}^{n}m_{j}^{3}\right),\label{RL1:eq1}\\
S_{(2,1,0,\dots,0)} = \dfrac{-1}{3}\left(\sum_{j=1}^{n}m_{j}^{3}\right),\label{RL1:eq2}\\
S_{(1,1,1,0\dots,0)} = \dfrac{1}{3}\left(\sum_{j=1}^{n}m_{j}^{3}\right)\label{RL1:eq3}.
\end{align}

\end{lemma}
\begin{proof}
The subspace should be identified with the Lie algebra $\mathfrak{t}$ of the maximal torus of $SU(n)$ defined by the diagonal matrices (here we ignore the factor of $\sqrt{-1}$).  Hence the restriction of the Schur polynomials yield degree three elements of  $\mathbb{R}[\mathfrak{t}]^{W}$.  As remarked previously in Section \ref{sec:3}, there is a unique generator in degree three and so any such polynomial is a multiple of some fixed non-zero element.  We choose  $\sum_{j=1}^{n}m_{j}^{3}$ as this generating element.\\
\\
To compute the multiples we simply need to evaluate each polynomial on an element of $\mathfrak{t}$ which is not a zero of $\sum_{j=1}^{n}m_{j}^{3}$. We choose the element 
$$m=(1,1,-2,0,\dots,0).$$
In this case, we note
$$
\sum_{j=1}^{n}m_{j} = -6, \qquad \sum_{1\leq i\neq j\leq n}m_{i}m_{j}^{2} =6 \qquad \sum_{1\leq i<j<l\leq n}m_{i}m_{j}m_{l} = -2.
$$
Thus 
$$S_{(3,0,0,\dots,0)}(m) = -2  \qquad S_{(2,1,0,\dots,0)}(m) = 2 \quad S_{(1,1,1,0\dots,0)}(m) = -2,
$$
and the result follows.
\end{proof}

We will also require the restricted version of Lemma \ref{AL2}.

\begin{lemma}\label{ResLem2}
Let $\dfrac{c_{3}}{(k(n-k)+3)!}$ be the coefficient of $t^{k(n-k)+3}$ in the power series expansion about $0$ of $\det(\mathcal{M}(t))$ and let $\varepsilon_{k,n}$ be as in Lemma \ref{AL1}. Then restricting the inputs $m$ to the subspace $\mathcal{T}_{n}$ yields
$$c_{3} = \varepsilon_{k,n} V \left(\prod_{i=1}^{3}(k(n-k)+i)\right)\dfrac{k(n-k)(n-2k)c_{0}}{3(n-2)(n-1)n(n+1)(n+2)}\left(\sum_{j=1}^{n}m_{j}^{3}\right).$$

\end{lemma} 
\begin{proof}
We substitute Equations (\ref{RL1:eq1}),(\ref{RL1:eq2}), and (\ref{RL1:eq3}) into the expression for $c_{3}$ in Lemma \ref{AL2}.  Simplifying 
$$c_{(3,0,0)}-c_{(2,1,0)}+c_{(1,1,1)}=$$
$$ \left(\prod_{i=1}^{3}(k(n-k)+i)\right)\dfrac{kc_{0}}{6n}\left( \dfrac{(k+1)(k+2)}{(n+1)(n+2)}-\dfrac{2(k-1)(k+1)}{(n-1)(n+1)}+\dfrac{(k-2)(k-1)}{(n-2)(n-1)}\right) = $$
$$
\left(\prod_{i=1}^{3}(k(n-k)+i)\right)\dfrac{k(n-k)(n-2k)c_{0}}{(n-2)(n-1)n(n+1)(n+2)}
$$

yields the result.

\end{proof}

\section{Proof of the Theorem \ref{ThmGr}} \label{sec:5}
The Lie algebra of $SU(n)$, $\mathfrak{su}(n)$, is identified with trace-free $n\times n$  skew-Hermitian matrices. The rank of $SU(n)$ is $n-1$ with a maximal torus $T$ being given by diagonal matrices. Hence the Lie algebra of $T$, $\mathfrak{t}$, can be identified with
$$ \textrm{Diag}(\sqrt{-1}\mu_{1},\sqrt{-1}\mu_{2},\dots,\sqrt{-1}\mu_{n}) \textrm{ with } \sum_{i}\mu_{i}=0.$$ 
The roots can identified with $e_{j}-e_{l}$ for $j\neq l$ where $e_{j}$ is the diagonal matrix with the entry $\sqrt{-1}$ in the $j^{th}$ coefficient and we are using the inner product $\langle X,Y\rangle = \tr(X^{\ast}Y)$ to identify $\mathfrak{su}(n)$ and $\mathfrak{su}^{\ast}(n)$. The Weyl group $W\cong \textrm{Sym}_{n}$ acts on $T$ by permuting the elements of the diagonal and $W$ acts on $\mathfrak{t}$ permuting the $\sqrt{-1}\mu_{i}$.\\
\\
The adjoint orbits we consider can be represented by an element $\xi \in \mathfrak{t}$. We let
 $$\xi = \textrm{Diag}(\underbrace{\sqrt{-1}\mu_{1},\dots,\sqrt{-1}\mu_{1}}_{k\text{ entries}},\underbrace{\sqrt{-1}\mu_{2}, \sqrt{-1}\mu_{2},\dots,\sqrt{-1}\mu_{2}}_{n-k\text{ entries}})$$
where $\mu_{1}>0$ and ${k\mu_{1}+(n-k)\mu_{2}=0}$. In fact, it will be useful to fix 
$$\mu_{1} = \frac{n-k}{n} \textrm{ and } \mu_{2} = -\frac{k}{n},$$
so that $\mu_{1}-\mu_{2}=1.$ We get an identification of $\mathcal{O}_{\xi}$ with $Gr_{k}(\mathbb{C}^{n})$ by considering the $k$-plane generated by the span of the $\sqrt{-1}\mu_{1}$ eigenspace at each point in the orbit.\\ 
\\
The vector $D\in \mathfrak{t}$ given by  $D=\textrm{Diag}(2\pi \sqrt{-1}m_{1}, 2\pi \sqrt{-1}m_{2},\dots,2\pi \sqrt{-1}m_{n})$, where  ${m_{j}\in \mathbb{Z}}$, $m_{j}\neq m_{l}$ for $j\neq l$, and $\sum_{j}m_{j}=0$, generates a circle action on $\mathcal{O}_{\xi}$ and is also regular.  Recall the function $f_{D}:\mathcal{O}_{\xi}\rightarrow\mathbb{R}$ defined by Equation (\ref{ef_def}). By Lemma \ref{eflem}, $f_{D}$ is an eigenfunction of the Laplacian and  by Lemma \ref{Bottlem}, the fixed points of the circle action (or equivalently the critical points of $f_{D}$) are the orbit of $\xi$ under the Weyl group $\textrm{Sym}_{n}$ which are the vectors consisting of to the $^{n}C_{k}$ possible placements of the $\mu_{1}$s. If we index a fixed point of the circle action by the $k$-element set ${J\subset \{1,2,\dots,n\}}$ corresponding to this placement and denoting such a fixed point $q_{J}$, then the value of the function $f_{D}$ at this point is
$$f_{D}(q_{J}) = 2\pi\left(\mu_{1}\sum_{j\in J}m_{j}+\mu_{2}\sum_{j\in J^{c}}m_{j}\right).$$
The set of roots $\alpha \in \Lambda_{R}$ such that $\langle \alpha,q_{J}\rangle>0$ are $e_{j_{1}}-e_{j_{2}}$ where $j_{1}\in J$ and $j_{2}\in J^{c}$.  Hence the weight of the induced action on the holomorphic tangent space at $q_{J}$ is
$$\varpi(q_{J}) = \prod_{j\in J, \ l\in J^{c}}(m_{j}-m_{l}).$$
The Duistermaat--Heckman formula yields
$$\int_{\mathcal{O}_{\xi}}e^{-tf_{D}}\omega^{n(n-k)} =  \frac{[k(n-k)]!}{t^{n(n-k)}}\sum_{J\subset \{1,2,\dots,n\} \ : \ |J|=k}\frac{e^{-\left(\mu_{1}\sum_{j\in J}m_{j}+\mu_{2}\sum_{j\in J^{c}}m_{j}\right)2\pi t}}{\prod_{j\in J, \ l\in J^{c}}(m_{j}-m_{l})}.$$
We manipulate this expression by pulling out the Vandermonde factor 
$$V=\prod_{1\leq j<l \leq n}(m_{j}-m_{l}),$$ so the sum can be written as
$$\frac{[k(n-k)]!}{Vt^{k(n-k)}}\sum_{J\subset \{1,2,\dots,n\}: \ |J|=k}\varepsilon_{J}e^{-\left(\mu_{1}\sum_{j\in J}m_{j}+\mu_{2}\sum_{j\in J^{c}}m_{j}\right)2\pi t}\prod_{j<l\in J}(m_{j}-m_{l})\prod_{j<l\in J^{c}}(m_{j}-m_{l}),$$
where $\varepsilon_{J}$ is the sign of the permutation sending $1,\dots,k$ to the sequence ${j_{1}<j_{2}<\dots<j_{k} \in J}$ and $k+1,k+2,\dots n$ to the elements of $J^{c}$.\\
\\
The sum is the (Laplace) expansion of the determinant of the following matrix
$$ M(t) = \left(
\begin{array}{cccc}
e^{-\mu_{1}m_{1}t} & e^{-\mu_{1}m_{2}t} & \dots  & e^{-\mu_{1}m_{n}t} \\ 
e^{-\mu_{1}m_{1}t}m_{1} & e^{-\mu_{1}m_{2}t}m_{2} & \dots  & e^{-\mu_{1}m_{n}t}m_{n}\\
\vdots &\vdots &\vdots &\vdots\\
e^{-\mu_{1}m_{1}t}m_{1}^{k-1} & e^{-\mu_{1}m_{2}t}m_{2}^{k-1} & \dots  & e^{-\mu_{1}m_{n}t}m_{n}^{k-1}\\ 
e^{-\mu_{2}m_{1}t} & e^{-\mu_{2}m_{2}t} & \dots  & e^{-\mu_{2}m_{n}t} \\ 
e^{-\mu_{2}m_{1}t}m_{1} & e^{-\mu_{2}m_{2}t}m_{2} & \dots  & e^{-\mu_{2}m_{n}t}m_{n} \\ 
\vdots &\vdots &\vdots &\vdots\\
e^{-\mu_{2}m_{1}t}m_{1}^{n-k-1} & e^{-\mu_{2}m_{2}t}m_{2} ^{n-k-1}& \dots  & e^{-\mu_{2}m_{n}t}m_{n}^{n-k-1} 
\end{array}
\right).$$
Hence, we have  
\begin{small}
\begin{equation}\label{DHeqDetM}
\int_{\mathcal{O}_{\xi}}e^{-tf_{D}}\omega^{k(n-k)} =\frac{[k(n-k)]!}{Vt^{k(n-k)}}\det(M(2\pi t)) = \varepsilon_{k,n}\frac{[k(n-k)]!}{Vt^{k(n-k)}}\det(\mathcal{M}(2\pi t)),
\end{equation}
\end{small}
where $\mathcal{M}(t)$ is the matrix (\ref{DHMatrix}) and $\varepsilon_{k,n}$ is the signed quantity from Lemma \ref{AL1}. To obtain the final equality we multiply the matrix by $e^{\mu_{2}\left(\sum_{j=1}^{n}m_{j}\right)t}$ then distribute this over each column. The equality follows as  $\sum_{j=1}^{n}m_{j}=0$ and $\mu_{1}-\mu_{2}=1$. \\
\\
We remark again, in order to be able to apply Kr\"oncke's test we need $f_{D}$ to be a genuine eigenfunction of the Laplacian. This means we must normalise $f_{D}$ so that it has mean value zero or, equivalently, so that the first derivative of 
$$\int_{\mathcal{O}_{\xi}}e^{-tf_{D}}\omega^{k(n-k)}$$
vanishes when $t=0$. The $t^{k(n-k)+1}$ coefficient of the power series expansion of $\mathcal{M}(t)$ is a multiple of $S_{(1,0,\dots,0)} = \sum_{j}m_{j}$ which vanishes and hence $f_{D}$ is an eigenfunction. As mentioned after the proof of Theorem \ref{ThmCST} in Section \ref{sec:3}, the function $f_{D}$ does not really need normalising as there are no homogenous, degree 1, $\mathrm{Sym}_{n}$-invariant polynomials by considering the $A_{n}$ part of the Chevalley--Shephard--Todd theorem. This was also directly shown in Lemma \ref{f_mv_lem}. \\
\\
The proofs of the main theorems are now straightforward. Using Equation (\ref{DHeqDetM}) we see that if we can find a regular $D\in \mathfrak{t}$, that generates a circle action, such that the third derivative of the quantity 
$$\varepsilon_{k,n}\frac{[k(n-k)]!}{Vt^{k(n-k)}}\det(\mathcal{M}(2\pi t)),$$
does not vanish when evaluated at $0$, then
$$\int_{\mathcal{O}_{\xi}} f_{D}^{3} \ \omega^{k(n-k)}\neq 0.$$
Thus we will be able to apply Kr\"oncke's stability criterion in Thereom \ref{KKstabthm} to obtain the result.\\
\\
It is clear that the third derivative of
$$\varepsilon_{k,n}\frac{[k(n-k)]!}{Vt^{k(n-k)}}\det(\mathcal{M}(2\pi t)),$$
 evaluated at $0$ is a non-zero multiple of the quantity $c_{3}$ given in Lemma \ref{ResLem2}. We see that $c_{3}$ does not vanish provided $n\neq 2k$ and 
$$\sum_{j=1}^{n}m_{j}^{3}\neq 0.$$
There are many possible choices available. For example, we could choose
$$D= 2\pi\sqrt{-1}\mathrm{Diag}\left(1,2,3,...,n-1, -\frac{n(n-1)}{2}\right).$$
In which case 
$$\sum_{j=1}^{n}m_{j}^{3} = \left(\frac{n(n-1)}{2}\right)^{2}-\left(\frac{n(n-1)}{2}\right)^{3} =\left(\frac{n(n-1)}{2}\right)^{2}\left(\frac{(n+2)(1-n)}{2}\right), $$
which is clearly non-zero for all $n> 2$.
 \bibliographystyle{acm} 
 \bibliography{HMWRefs}

\end{document}